\documentclass[12pt]{amsart}

\usepackage{amsmath,amssymb,amsthm}
\usepackage[margin=1in]{geometry}
\usepackage{setspace}

\newtheorem{theorem}{Theorem}[section]
\newtheorem{proposition}[theorem]{Proposition}
\newtheorem{lemma}[theorem]{Lemma}

\theoremstyle{definition}
\newtheorem{definition}[theorem]{Definition}
\theoremstyle{remark}
\newtheorem{remark}[theorem]{Remark}
\newtheorem{example}[theorem]{Example}

\newcommand{\define}{\overset{\mathrm{def}}{=}}
\renewcommand{\b}{\mathrm{b}}
\renewcommand{\L}{\mathrm{L}}
\renewcommand{\Im}{\mathrm{Im}}
\renewcommand{\Re}{\mathrm{Re}}

\numberwithin{equation}{section}

\begin{document}
\title[LARGE-SCALE VOLUME GROWTH]{On uniform large-scale volume growth for the Carnot-Carath\'eodory metric on unbounded model hypersurfaces in $\mathbb{C}^2$}

\author[E. DLUGIE]{Ethan Dlugie}
\address{Department of Mathematics, Northwestern University, 721 University Place, Evanston, IL 60201, USA}
\email{EthanDlugie2018@u.northwestern.edu}
\thanks{Dlugie's work was supported by a grant from the WCAS Undergraduate Research Grant Program which is administered by Northwestern University's Weinberg College of Arts and Sciences.}

\author[A. PETERSON]{Aaron Peterson}
\address{Department of Mathematics, Northwestern University, 2033 Sheridan Road, Evanston, IL 60208, USA}
\email{aaron@math.northwestern.edu}
\thanks{The authors thank the referee for many suggestions which improved the quality and clarity of the paper.}

\keywords{Carnot-Carath\'eodory metric, global behavior, volume growth}
\subjclass[2010]{Primary 53C17; Secondary 32V15, 43A85}

\begin{abstract}
	We consider the rate of volume growth of large Carnot-Carath\'eodory metric balls on a class of unbounded model hypersurfaces in $\mathbb{C}^2$. When the hypersurface has a uniform global structure, we show that a metric ball of radius $\delta \gg 1$ either has volume on the order of $\delta^3$ or $\delta^4$. We also give necessary and sufficient conditions on the hypersurface to display either behavior.
\end{abstract}
\maketitle

\section{Introduction}
	The study of holomorphic functions on pseudoconvex domains $\Omega\subseteq\mathbb{C}^n$ ($n\geq 2$) often reduces to studying the partial differential operator $\bar{\partial}$ on $\Omega$ given by $\bar{\partial}(f)=\sum f_{\bar{z}_j}d\bar{z}^j$. We can study the boundary values of holomorphic functions (on ${\rm b}\Omega$) by studying the partial differential operator $\bar{\partial}_b$ induced on ${\rm b}\Omega$ by $\bar{\partial}$. We locally express $\bar{\partial}_b$ in terms of differentiation with respect to $(n-1)$-antiholomorphic vector fields (the so-called Cauchy-Riemann, or CR, vector fields on ${\rm b}\Omega$) that are tangent to ${\rm b}\Omega$. Under mild non-degeneracy conditions on ${\rm b}\Omega$ we can access a family of metrics on ${\rm b}\Omega$ specifically adapted to the study of $\bar{\partial}$ and $\bar{\partial}_b$, in the sense that they capture important geometric aspects of ${\rm b}\Omega$. One of these, the Carnot-Carath\'eodory (CC) metric $d(\mathbf{p},\mathbf{q})$, measures the infimal length of paths on ${\rm b}\Omega$ that not only connect the points $\mathbf{p}$ and $\mathbf{q}$, but are also almost-everywhere tangent to the real and imaginary parts of the CR vector fields; see \cite{Street2014} and the references therein for an extensive history of this metric and its applications to the study of $\bar{\partial}$ and $\bar{\partial}_b$.
	\smallskip
	
	In this paper we consider the CC metric $d(\mathbf{p},\mathbf{q})$ induced on the boundary of a model pseudoconvex domain $\Omega \subset  \mathbb{C}^2$ by the real and imaginary parts of the CR vector field on $\b\Omega$. In particular, we seek to understand the volume growth of the metric balls $B_d(\mathbf{p},\delta)$ when $\Omega$ is of the form
	\begin{equation*}
		\Omega=\{ (z_1,z_2)\in \mathbb{C}^2\ :\ {\rm Im}(z_2)>P(z_1)\},
	\end{equation*}
	where $P:\mathbb{C}\to \mathbb{R}$ is smooth, subharmonic, and non-harmonic. Under mild non-degeneracy conditions on $\Delta P$ is it known (\cite{Montanari2012,Nagel1988,Nagel1989,Nagel1985}) that for $\delta \leq 1$ the metric ball $B_d(\mathbf{p},\delta)$ is comparable to a `shorn' or `twisted' ellipsoid with radius $\delta$ in the directions spanned by the real and imaginary parts of the CR vector field and radius $\Lambda((z_1,z_2),\delta)$ in the $\Re(z_2)$-direction.  If we equip $\b\Omega$ with the Lebesgue measure $dm(z,t)$ that it receives via its identification with $\mathbb{C} \times \mathbb{R}$ given by $(z_1,z_2) \mapsto (z,t)$ where $z=z_1=x+iy$ and $t=\Re(z_2)$, then this small CC metric ball has volume comparable to that of the twisted ellipsoid:
	\begin{equation}
		Vol(B_d(\mathbf{p},\delta))\approx \delta^2 \Lambda(\mathbf{p},\delta). \label{vol}
	\end{equation}
	We build on the earlier work  of the second author \cite{Peterson2014} which sought to understand the possible rate of growth of $Vol(B_d(\mathbf{p},\delta))$ for model domains $\Omega$ such that when $\delta$ is large, the Euclidean radius
	\begin{equation*}
		\Lambda((z_1,z_2),\delta)=\sup \{ |\Re(z_2'-z_2)|\ :\ d((z_1,z_2),(z_1,z_2'))<\delta \}
	\end{equation*}
	of $B_d((z_1,z_2),\delta)$ in the $\Re(z_2)$-direction is essentially independent of $(z_1,z_2)$. The quantity $\Lambda(\mathbf{p},\delta)$ is called the \textit{global structure} of $\b\Omega$, and we make precise the $(z_1,z_2)$-independence condition described above with the following definition.
	\begin{definition}
		If there exists $\delta_0 > 0$, a function $f : [\delta_0, +\infty) \to [0,+\infty)$, and positive constants $0<c<C<+\infty$ such that $cf(\delta) \leq \Lambda(\mathbf{p}, \delta) \leq Cf(\delta)$ for all $\delta \geq \delta_0$ and $\mathbf{p} \in \b\Omega$, then we say that $(f(\delta),\delta_0)$ is a \textit{uniform global structure} or UGS for $\b\Omega$.
	\end{definition}
	For such domains $\Omega$ we also have \eqref{vol} when $\delta$ is large (see Remark \ref{Remark vol}), and therefore the volume growth of CC metric balls of any size is completely understood once we understand $\Lambda(\mathbf{p},\delta)$ for large $\delta$.\smallskip
		
	\begin{example}{\rm In \cite{Nagel1988}, it is shown that when $P(z_1)$ is a subharmonic, nonharmonic polynomial (and where $\Delta P$ has degree $m-2$), $$\Lambda((z_1,z_2),\delta) \approx \sum_{k=0}^{m-2} \Bigg(\sum_{\alpha=0}^k \Big| \frac{\partial^k \Delta P}{\partial z_1^\alpha \partial \bar{z}_1^{k-\alpha}}(z_1)\Big|\Bigg) \delta^{k+2}.$$
	
	In particular, when $P(z_1)=|z_1|^2$ (so that $\Delta P(z_1)\equiv 4$) we have $\Lambda((z_1,z_2),\delta)\approx 4\delta^2$, and therefore $(\delta^2,1)$ is a uniform global structure for ${\rm b}\Omega$.
	
	On the other hand, if $P(z_1)=|z_1|^4$, then $\Lambda((z_1,z_2),\delta)\approx |z_1|^2\delta^2+|z_1|\delta^3+\delta^4\approx (|z_1|+\delta)^2\delta^2$, and therefore is not uniform in $z_1\in \mathbb{C}$.  This shows that ${\rm b}\Omega$ has no uniform global structure.  More generally, if $P$ is a subharmonic, non-harmonic polynomial, then ${\rm b}\Omega$ does not have a uniform global structure when $\Delta P$ is not constant.
	}\end{example}	
	
	The following result from \cite{Peterson2014} controls the growth of uniform global structures.
	\begin{theorem}[\cite{Peterson2014} Theorem 1.2]
		If $\b\Omega$ has a UGS $(f(\delta),\delta_0)$, then there are positive constants $0<c<C<+\infty$ such that $c\delta \leq f(\delta) \leq C\delta^2$ for all $\delta \geq \delta_0$. \label{thm 1.2.}
	\end{theorem}
	So when $\b\Omega$ has a UGS and $\delta \gg 1$, the global structure at any point grows at least linearly and at most quadratically in $\delta$. Examples are given in \cite{Peterson2014} where $\b\Omega$ has a UGS linear in $\delta$ and quadratic in $\delta$. Our question is whether there exist examples where the UGS grows somewhere ``between" linear and quadratic. For example, are there examples for $\b\Omega$ with UGS $(\delta^{\frac{3}{2}},\delta_0)$ or $(\delta \log\delta, \delta_0)$?\smallskip

	\begin{example}{\rm To see that this question is not trivial, fix $\alpha\in \left(0,\frac{2}{3}\right)$ and choose a subharmonic function $P:\mathbb{C}\to\mathbb{R}$ such that $\Delta P(z)=(1+|z|^2)^{-\alpha/2}$. Using our techniques and those of \cite{Peterson2014} one can show that there exist constants $0<c<C<+\infty$ such that for all $\delta>0$, $$c\delta^{2-\alpha}\leq \Lambda((0,0),\delta)\leq C\delta^{2-\alpha}\quad\mbox{and}\quad \Lambda((\delta^{\frac{3}{2}},0),\delta)\leq C\delta^{2-3\alpha/2}.$$
	Thus $\Lambda((0,0),\delta)$ grows at a rate comparable to $\delta^{2-\alpha}$, but $\Lambda((\delta^{\frac{3}{2}},0),\delta)$ grows no faster than $\delta^{2-3\alpha/2}$. This illustrates that it is possible for the global structure to grow (in $\delta$) at non-polynomial rates, but (since $\alpha<\frac{3}{2}\alpha$) not necessarily uniformly in the base point $(z_1,z_2)$.
	}\end{example}
	
	Our first main theorem (proven in Section \ref{S dichot}) answers our question negatively.
	\begin{theorem}
		If $\b\Omega$ has UGS $(f(\delta),\delta_0)$, then either $(\delta^2,\delta^*)$ or $(\delta,\delta^*)$ is a UGS for $\b\Omega$ for some $\delta^* > 0$. \label{dichotomy}
	\end{theorem}
	We subsequently give necessary and sufficient conditions on $\b\Omega$ for both linear (Theorem \ref{linear}) and quadratic (Theorem \ref{quadratic}) growth of the UGS, thereby completely describing the conditions under which any particular model domain has a uniform global structure.\smallskip
		
	The volume growth of CC metric balls in model domains $\Omega$ as above for large $\delta$ is only explicitly understood when $P$ is a subharmonic, non-harmonic polynomial \cite{Nagel1988} or in the limited examples considered in \cite{Peterson2014} mentioned above. In some situations one can obtain upper bounds for the rate of volume growth (see \cite{Chang2014}), but one cannot hope for precise control of $Vol(B_d(\mathbf{p},\delta))$ for general $P$. On the other hand, applications of volume growth estimates are many and varied; for example, one can use these estimates to identify spaces of homogeneous type \cite{Coifman1977}, study singular integral operators \cite{Street2014}, and even to decide whether or not the boundaries of two model domains are quasi-conformally equivalent \cite{Faessler2014,Heinonen1998}.\smallskip
	
	Our paper is structured as follows. Section \ref{S prelim} gives relevant definitions and notation that will be used extensively throughout the paper and recalls past results. In Section \ref{S altdesc} we gain some intuition about how a UGS behaves and prove a key and explicit alternative characterization of the UGS. In Section \ref{S dichot} we prove Theorem \ref{dichotomy}, followed in Section \ref{S identify} by necessary and sufficient conditions for a given model domain to possess a uniform global structure. Section \ref{S future} concludes the paper and offers future directions of study.
	
\section{Preliminaries} \label{S prelim}
	With $\Omega$ as in the introduction, the space of tangential CR vector fields on $\b\Omega$ is spanned by
	\begin{equation*}
		\bar{Z} = 2 \frac{\partial}{\partial\bar{z}_1} - 4i P_{\bar{z}_1}(z_1) \frac{\partial}{\partial\bar{z}_2}.
	\end{equation*}
	We identify $\b\Omega$ with $\mathbb{C} \times \mathbb{R}$ via the diffeomorphism $(z_1,z_2) \mapsto (z,t) \in \mathbb{C} \times \mathbb{R}$ where $z = z_1 = x + iy$ and $t = \Re(z_2)$. Under this transformation, $\bar{Z}$ becomes
	\begin{equation*}
		\bar{Z} = 2\frac{\partial}{\partial \bar{z}} - 2iP_{\bar{z}}(z)\frac{\partial}{\partial t} = \left( \frac{\partial}{\partial x} + P_y(x,y) \frac{\partial}{\partial t} \right) - i\left( -\frac{\partial}{\partial y} + P_x(x,y) \frac{\partial}{\partial t} \right) \define X-iY.
	\end{equation*}
	As stated in the introduction, we give $\b\Omega$ the Lebesgue measure $dm(z,t)$ that it receives upon identification with $\mathbb{C} \times \mathbb{R}$. For the rest of the paper, we work on $\mathbb{C} \times \mathbb{R}$ instead of $\b\Omega$ to simplify notation.\smallskip
	
	We define the CC metric $d:(\mathbb{C} \times \mathbb{R}) \times (\mathbb{C} \times \mathbb{R}) \to [0,+\infty)$ by
	\begin{align}
		d(\mathbf{p},\mathbf{q}) = \inf \Big\{ \delta > 0\ :\ \exists &\gamma:[0,1] \to \mathbb{C}\times\mathbb{R}, \gamma(0)=\mathbf{p}, \gamma(1)=\mathbf{q},\notag\\
		& \gamma'(s) = \delta \alpha(s)X(\gamma(s)) + \delta \beta(s)Y(\gamma(s))\ a.e.,\notag\\
		& \alpha,\beta \in FPWS[0,1],\ |\alpha(s)|^2+|\beta(s)|^2<1\ a.e. \Big\}. \label{metric}
	\end{align}
	Here $FPWS[0,1]$ (read ``finite piecewise smooth") denotes the set of functions $f:[0,1] \to \mathbb{R}$ which are smooth except at a finite number of points and whose derivatives extend continuously to those points from each side separately.\smallskip
	
	The global structure $\Lambda((z,t),\delta)$, the radius in the $t$-direction of the CC ball, is then defined as
	\begin{equation}
		\Lambda((z,t),\delta) \define \sup \left\{ |t'-t|\ :\ d( (z,t),(z,t') ) < \delta \right\}. \label{glbl struc def}
	\end{equation}
	Note that the quantity \eqref{glbl struc def} is actually \textit{independent} of the $t$-coordinate because the solutions to the differential equation in \eqref{metric} are translation invariant in $t$. To simplify notation, we will therefore write $\Lambda(z,\delta)$ instead of $\Lambda((z,t),\delta)$ for the remainder of the paper, treating $\Lambda$ as a function from $\mathbb{C} \times (0,+\infty) \mapsto [0,+\infty)$. The first observation of \cite{Peterson2014} is that definition \eqref{glbl struc def} is in fact equivalent to the following statement in terms of curves in $\mathbb{C}$, independent of $t$:
	\begin{align}
		\Lambda(z,\delta) = \sup \Big\{ \oint_\gamma P_y(\gamma)dx-P_x(\gamma)dy\ :\ & \gamma:[0,1] \to \mathbb{C}, \gamma(0)=\gamma(1)=z, |\gamma'(s)| \leq \delta\ a.e.,\notag\\
		& \gamma'(s)=\alpha(s)+i\beta(s),\ \alpha,\beta \in FPWS[0,1] \Big\}. \label{alt glbl struc def}
	\end{align}
	We write $L(\gamma)=\int_a^b |\gamma'(s)|ds$ for the usual Euclidean length of a piecewise smooth curve $\gamma:[a,b] \to \mathbb{C}$. The following geometric definition from \cite{Peterson2014} will be essential to our understanding of global structures.
	\begin{definition}
		We say $A \subset \mathbb{C}$ is a \textit{pen} if $A$ is open, connected, simply connected, and if $\b A$ can be parametrized by a continuous piecewise smooth curve $\gamma:[0,1]\to\mathbb{C}$ with $\gamma'(s)=\alpha(s)+i\beta(s)$ where $\alpha,\beta \in FPWS[0,1]$. We call $\L(\b A)=\L(\gamma)$ the amount of \textit{fencing} used to enclose A. For a fixed $z \in \mathbb{C}$ and $\delta>0$, we say that a finite collection of pens $R=(R_1, \dotsc, R_N)$ is a \textit{$(z, \delta)$-stockyard} if
		\begin{center}
			$\displaystyle{z \in \bigcup_{i=1}^N \b R_i}$, \quad
			$\displaystyle{\sum_{i=1}^N \L(\b R_i) \leq \delta}$, \quad and \quad
			$\displaystyle{\bigcup_{i=1}^N \b R_i}$ is connected.
		\end{center}
	\end{definition}
	\begin{remark}
		We will often use in this paper the fact that given a pen $A$, $A \subseteq B(z,\L(\b A))$ for any point $z \in A$, where $B(z,\rho)$ denotes the open Euclidean disc in $\mathbb{C}$ of radius $\rho$ centered at $z$.\label{Remark penball}
	\end{remark}
	Thinking of global structures in terms of \eqref{alt glbl struc def}, \cite{Peterson2014} provides the following theorem.
	\begin{theorem}[\cite{Peterson2014} Theorem 1.1]
		$\Lambda(z,\delta) = \sup\limits_{(z,\delta)-stockyards\ R}\ {\displaystyle \sum\limits_{R_i \in R}\ \int\limits_{R_i} \Delta P(w)dm(w)}$. \label{thm 1.1.}
	\end{theorem}
	Here $dm(\cdot)$ denotes the Lebesgue measure on $\mathbb{C}$. The problem of calculating the global structure, an inherently three dimensional problem, is therefore reduced to a question in two dimensions. Furthermore, notice that because $P$ was assumed to be subharmonic and non-harmonic, we can think of $\Delta P$ as a density function in the plane. In this context, integration over a pen measures the `mass' of the region covered by the pen, and integration over a stockyard is then the sum of the mass collected by the individual pens. The global structure $\Lambda(z,\delta)$ is then just the most mass one can collect with a stockyard touching $z$ constructed with at most $\delta$ amount of fencing.\smallskip
		
	To simplify notation in our estimates, we introduce the following notation. For two non-negative quantities $A$ and $B$, we write $A \lesssim B$ (read ``$A$ is controlled above by $B$") if there exists some constant $c>0$, independent of all relevant quantities, such that $A \leq cB$. We say $A \gtrsim B$ (read ``$A$ is controlled below by $B$") if $B \lesssim A$, and $A \approx B$ (read ``$A$ is comparable to $B$") if both $A \lesssim B$ and $B \lesssim A$.

\section{Alternate description of uniform global structures} \label{S altdesc}
	When $\b\Omega$ has a UGS $(f(\delta),\delta_0)$ and when $\delta \geq \delta_0$, we expect that for every point $z$ in the plane we can find a high density region whose distance from the point is no more than $\delta$. We should then be able to construct a $(z,N\delta)$-stockyard for an appropriately fixed natural number $N$ which covers this region with one or more pens. Otherwise $\Lambda(z,\delta)$ would be vanishingly small at certain points. We also expect that no point should be within $\delta$ of a region of exceedingly high density. Otherwise $\Lambda(z,\delta)$ would be exceedingly large at certain points. Before we make this notion precise in Proposition \ref{alt UGS} of this section, we need two lemmas.\smallskip
	
	A simple observation about one formula for a UGS is the following.
	\begin{lemma}
		If $\b\Omega$ has UGS $(f(\delta),\delta_0)$, then $\left(\sup_{z \in \mathbb{C}} \Lambda(z,\delta), \delta_0\right)$ is also a UGS for $\b\Omega$. \label{sup UGS}
	\end{lemma}
	\begin{proof}
		Fix some $z \in \mathbb{C}$. By the definition of UGS, there exist constants $c,C>0$ independent of $z$ and $\delta$ such that
		\begin{equation*}
			cf(\delta) \leq \Lambda(z,\delta) \leq Cf(\delta).
		\end{equation*}
		So $Cf(\delta)$ is an upper bound for $\{ \Lambda(z,\delta) : z \in \mathbb{C} \}$, which gives $\sup\limits_{z \in \mathbb{C}} \Lambda(z,\delta) \leq Cf(\delta)$	since the supremum is the least upper bound. Also $\sup\limits_{z \in \mathbb{C}} \Lambda(z,\delta) \geq \Lambda(z,\delta) \geq cf(\delta)$. So then
		\begin{equation*}
			\Lambda(z,\delta) \leq Cf(\delta) \leq \frac{C}{c} \sup\limits_{z \in \mathbb{C}} \Lambda(z,\delta),
		\end{equation*}
		and
		\begin{equation*}
			\Lambda(z, \delta) \geq c f(\delta) \geq \frac{c}{C} \sup\limits_{z \in \mathbb{C}} \Lambda(z,\delta)
		\end{equation*}
		for all $\delta \geq \delta_0$. Therefore $\left( \sup\limits_{z \in \mathbb{C}} \Lambda(z,\delta), \delta_0 \right)$ is a UGS for $\b\Omega$.
	\end{proof}
	Lemma \ref{sup UGS} makes it clear that we can take $f(\delta)$ to be a monotonically increasing function of $\delta$. We next show that $f(\delta)$ does not increase too quickly in the sense that if we double the amount of fencing available to construct stockyards, then the amount of mass one can collect should not grow exceedingly fast.
	\begin{lemma}
		If $\b\Omega$ has UGS $(f(\delta),\delta_0)$ then $f(\delta) \approx f(2\delta)$ for all $\delta \geq \delta_0$, with constants independent of $\delta$. \label{UGS growth}
	\end{lemma}
	\begin{proof}
		By Lemma \ref{sup UGS} we can without loss of generality take $f(\delta) = \sup\limits_{z \in \mathbb{C}} \Lambda(z,\delta)$. For if $(g(\delta),\delta_0)$ is any other UGS for $\b\Omega$ and we can prove the lemma for $f(\delta)$, then $g(\delta) \approx f(\delta) \approx f(2\delta) \approx g(2\delta)$. We prove first that $f(2\delta) \approx f(3\delta)$ for large $\delta$ and will show at the end of the proof that this is sufficient to establish the lemma.\smallskip
		
		Because $f(\delta)$ is a nondecreasing function, we trivially have $f(2\delta) \leq f(3\delta)$. We need only show then that $f(3\delta) \lesssim f(2\delta)$. To this end, fix $z_0 \in \mathbb{C}$ and $\delta \geq \frac{2}{3}\delta_0$, and let $R$ be any arbitrary $(z_0,3\delta)$-stockyard. There is a FPWS curve $\gamma:[0,1] \to \mathbb{C}$ with $\gamma(0) = \gamma(1) = z_0$, $\L(\gamma) \leq 3\delta$, and
		\begin{equation*}
			\sum_{R_i \in R} \int_{R_i} \Delta P(w)dm(w) = \oint_\gamma P_ydx-P_xdy.
		\end{equation*}
		We now produce seven continuous, piecewise smooth curves $\gamma_k : [0,1] \to \mathbb{C}$, $k=1,\dotsc,7,$ with $L(\gamma_k) \leq 2\delta$ and $\gamma_k'(s)=\alpha_k(s)+i\beta_k(s)$ with $\alpha_k,\beta_k \in FPWS[0,1]$ such that
		\begin{equation*}
			\oint_\gamma P_ydx - P_xdy = \sum_{k=1}^7 \oint_{\gamma_k} P_ydx - P_xdy.
		\end{equation*}
		Without loss of generality, suppose that $\gamma$ has constant speed so that
		\begin{equation}		
			\int_0^\frac{1}{3} |\gamma'(s)| ds = \int_\frac{1}{3}^\frac{2}{3} |\gamma'(s)| ds =  \int_\frac{2}{3}^1 |\gamma'(s)| ds \leq \delta. \label{length}
		\end{equation} 
		For convenience, we define $z_1 = \gamma(\frac{1}{3}), z_2 = \gamma(\frac{2}{3})$, and $z_3 = \gamma(1) = z_0$. We also denote by $\overrightarrow{z, w}$ the directed line segment from $z$ to $w$.\smallskip
		
		Now we have
		\begin{align}
			\oint_\gamma P_ydx-P_xdy & = \int_{\gamma[0,\frac{1}{3}]} P_ydx-P_xdy + \int_{\gamma[\frac{1}{3},\frac{2}{3}]} P_ydx-P_xdy + \int_{\gamma[\frac{2}{3},1]} P_ydx-P_xdy \notag\\
			& \hspace{5em} + \int_{\overrightarrow{z_0,z_1}} P_ydx-P_xdy + \int_{\overrightarrow{z_1,z_2}} P_ydx-P_xdy + \int_{\overrightarrow{z_2,z_3}} P_ydx-P_xdy \notag\\
			& \hspace{5em} + \int_{\overrightarrow{z_1,z_0}} P_ydx-P_xdy + \int_{\overrightarrow{z_2,z_1}} P_ydx-P_xdy + \int_{\overrightarrow{z_3,z_2}} P_ydx-P_xdy \notag\\
			& = \oint\limits_{\gamma[0,\frac{1}{3}]+\overrightarrow{z_1,z_0}} P_ydx-P_xdy + \oint\limits_{\gamma[\frac{1}{3},\frac{2}{3}]+\overrightarrow{z_2,z_1}} P_ydx-P_xdy + \oint\limits_{\gamma[\frac{2}{3},1]+\overrightarrow{z_3,z_2}} P_ydx-P_xdy \notag\\
			& \hspace{7em} + \oint\limits_{\overrightarrow{z_0,z_1} + \overrightarrow{z_1,z_2} + \overrightarrow{z_2,z_3}} P_ydx-P_xdy. \label{intsum}
		\end{align}
		We consider the contours of integration in each integral.\smallskip
		
		We define $\gamma_i = \gamma[\frac{i-1}{3},\frac{i}{3}]+\overrightarrow{z_i,z_{i-1}}$ for $i=1,2,3$. By \eqref{length}, the length of each contour $\gamma[\frac{i-1}{3}, \frac{i}{3}]$ is no more than $\delta$. And as the straight line between the endpoints of these contours, each directed line segment $\overrightarrow{z_i,z_{i-1}}$ also has length no more than $\delta$. In other words each $\gamma_i$ for $i=1,2,3$ is a closed curve of length no more than $2\delta$.\smallskip
		
		The last integral in \eqref{intsum} is taken over a closed contour composed of three line segments, each of length no more than $\delta$. For each $j=0,1,2$ define $b_j = \frac{1}{2}(z_j+z_{j+1})$ to be the bisector of segment $\overrightarrow{z_j,z_{j+1}}$, and for convenience define $b_{-1}=b_2$. We then define $\gamma_{j+4} = \overrightarrow{z_j,b_j} + \overrightarrow{b_j,b_{j-1}} + \overrightarrow{b_{j-1},z_j}$ and define $\gamma_7 = \overrightarrow{b_0,b_1} + \overrightarrow{b_1,b_2} + \overrightarrow{b_2,b_0}$. Then we have
		\begin{equation*}
			\oint\limits_{\overrightarrow{z_0,z_1} + \overrightarrow{z_1,z_2} + \overrightarrow{z_2,z_3}} P_ydx-P_xdy = \sum_{k=4}^7~ \oint\limits_{\gamma_k} P_ydx-P_xdy.
		\end{equation*}
		But by similar triangles
		\begin{equation*}
			\L(\gamma_k) = \frac{1}{2} \L(\overrightarrow{z_0,z_1} + \overrightarrow{z_1,z_2} + \overrightarrow{z_2,z_3}) \leq \frac{3}{2}\delta
		\end{equation*}
		for each $k=4,5,6,7$. Combining these observations with \eqref{intsum} and \eqref{alt glbl struc def}, we have
		\begin{align*}
			\sum_{R_i \in R} \int_{R_i} \Delta P(w)dm(w) & = \sum_{k=1}^7 \oint\limits_{\gamma_k} P_ydx-P_xdy\\
			& \leq \sum_{k=1}^7 \Lambda(\gamma_k(0),\L(\gamma_k)) \leq 3f(2\delta) + 4f\left(\frac{3}{2}\delta\right) \leq 7f(2\delta).
		\end{align*}
		for all $(z_0, 3\delta)$-stockyards $R$. Therefore by Theorem \ref{thm 1.1.} we see $\Lambda(z,3\delta) \leq 7f(2\delta)$ for all $z \in \mathbb{C}$, hence
		\begin{equation*}
			f(3\delta) = \sup\limits_{z \in \mathbb{C}} \Lambda(z, 3\delta) \leq 7f(2\delta).
		\end{equation*}
		In summary, for all $\delta \geq \frac{2\delta_0}{3}$ we have
		\begin{equation}
			f(2\delta) \leq f(3\delta) \leq 7f(2\delta) \label{2to3}.
		\end{equation}
		Now fix $\delta \geq \delta_0$. Because $f(\delta)$ is a nondecreasing function, we also trivially have $f(\delta) \leq f(2\delta)$. But by monotonicity and \eqref{2to3} we see
		\begin{equation*}
			f(2\delta) \leq f\left(\frac{9}{4}\delta\right) \leq 49f(\delta).
		\end{equation*}
		Therefore, $f(\delta) \approx f(2\delta)$ for all $\delta \geq \delta_0$.
	\end{proof}
	\begin{remark} \label{Remark vol}
		Lemma \ref{UGS growth} was used implicitly in \cite{Peterson2014} without proof or statement. The arguments of \cite{Peterson2014} show that for any fixed $z \in \mathbb{C}$
		\begin{equation*}
			\left\{ (w,s) \in \mathbb{C} \times \mathbb{R}\ :\ |w-z| < \frac{\delta}{4}, |s-t-T(z,w)|<\Lambda\left(z,\frac{\delta}{4}\right) \right\} \subseteq B_d\Big((z,t),\delta\Big)
		\end{equation*}
		and
		\begin{equation*}
			B_d\Big((z,t),\delta\Big) \subseteq \bigg\{ (w,s) \in \mathbb{C} \times \mathbb{R}\ :\ |w-z| < 3\delta, |s-t-T(z,w)|<\Lambda(z,3\delta) \bigg\},
		\end{equation*}
		where $T(z,w)=-2\Im\left(\int_0^1 (w-z)P_z(r(w-z)+z)dr\right)$ is the `twist' of the CC ball. Lemma \ref{UGS growth} then yields the formula $Vol(B_d((z,t),\delta)) \approx \delta^2 \Lambda(z,\delta)$ for $\delta \geq \delta_0$ when $\b\Omega$ has UGS $(f(\delta),\delta_0)$. This shows that we can think of $B_d((z,t),\delta)$ as a `twisted' ellipsoid in the case of large $\delta$, not just small $\delta$ as in \eqref{vol}.
	\end{remark}
	We are now ready to make precise the intuition laid out in the beginning of this section.
	\begin{proposition}
		If $\b\Omega$ has UGS $(f(\delta),\delta_0)$, then
		\[ \Lambda(z,\delta) \approx \sup\limits_{\hat{z} \in B(z,\delta)} \sup\limits_{0 < \hat{\delta} \leq \delta} \frac{\delta}{\hat{\delta}} \int_{B(\hat{z},\hat{\delta})} \Delta P(w)dm(w) \]
		uniformly for $z \in \mathbb{C}$ and $\delta \geq \delta_0$. \label{alt UGS}
	\end{proposition}
	\begin{proof}
		As in the proof of Lemma \ref{UGS growth} we assume without loss of generality that $f(\delta)$ is a nondecreasing function. For any choice of $\hat{z} \in B(z,\delta)$ and $0 < \hat{\delta} \leq \delta$, define a $(z,4\pi\delta)$-stockyard $R=(R_0, R_1, \dotsc, R_N)$ composed of one pen $R_0$ which is a circle touching $z$ and some point on $\b B(\hat{z},\hat{\delta})$ and $N = \left\lfloor \frac{\delta}{\hat{\delta}} \right\rfloor$ copies of $B(\hat{z},\hat{\delta})$. Using the fact that $\left\lfloor \frac{\delta}{\hat{\delta}} \right\rfloor \geq \frac{\delta}{2\hat{\delta}}$ because $\delta \geq \hat{\delta} > 0$, we have
		\begin{align*}
			\Lambda(z,\delta) \approx f(\delta) \approx f\left(16\delta\right) \geq f(4\pi\delta) & \gtrsim \Lambda(z,4\pi\delta)\\
			& \geq \sum_{R_i \in R} \int_{R_i} \Delta P(w)dm(w)\\
			& \geq \left\lfloor \frac{\delta}{\hat{\delta}} \right\rfloor \int_{B(\hat{z},\hat{\delta})} \Delta P(w)dm(w)\\
			& \geq \frac{\delta}{2\hat{\delta}} \int_{B(\hat{z},\hat{\delta})} \Delta P(w)dm(w).
		\end{align*}
		Therefore
		\begin{equation*}
		\Lambda(z,\delta) \gtrsim \sup\limits_{\hat{z} \in B(z,\delta)} \sup\limits_{0 < \hat{\delta} \leq \delta} \frac{\delta}{\hat{\delta}} \int_{B(\hat{z},\hat{\delta})} \Delta P(w)dm(w).
		\end{equation*}
		Now let $R=(R_1, \dotsc, R_M)$ be an arbitrary $(z,\delta)$-stockyard. For $i=1,\dotsc,M$ fix some point $z_i \in R_i$. Then recalling Remark \ref{Remark penball} we have
		\begin{align*}
			\sum_{R_i \in R} \int_{R_i} \Delta P(w)dm(w) & \leq \sum_{i=1}^M \int_{B(z_i, \L(\b R_i))} \Delta P(w)dm(w)\\
			& = \sum_{i=1}^M \frac{\L(\b R_i)}{\delta} \frac{\delta}{\L(\b R_i)}\int_{B(z_i, \L(\b R_i))} \Delta P(w)dm(w)\\
			& \leq \sum_{i=1}^M \left( \frac{\L(\b R_i)}{\delta} \right) \sup\limits_{\hat{z} \in B(z,\delta)} \sup\limits_{0 < \hat{\delta} \leq \delta} \frac{\delta}{\hat{\delta}} \int_{B(\hat{z},\hat{\delta})} \Delta P(w)dm(w)\\
			& \leq \frac{\delta}{\delta} \sup\limits_{\hat{z} \in B(z,\delta)} \sup\limits_{0 < \hat{\delta} \leq \delta} \frac{\delta}{\hat{\delta}} \int_{B(\hat{z},\hat{\delta})} \Delta P(w)dm(w)\\
			& = \sup\limits_{\hat{z} \in B(z,\delta)} \sup\limits_{0 < \hat{\delta} \leq \delta} \frac{\delta}{\hat{\delta}} \int_{B(\hat{z},\hat{\delta})} \Delta P(w)dm(w).
		\end{align*}
		Therefore 
		\begin{equation*}
			\Lambda(z,\delta) \leq \sup\limits_{\hat{z} \in B(z,\delta)} \sup\limits_{0 < \hat{\delta} \leq \delta} \frac{\delta}{\hat{\delta}} \int_{B(\hat{z},\hat{\delta})} \Delta P(w)dm(w),
		\end{equation*}
		completing the proof.
	\end{proof}
	
\section{Proof of Theorem \ref{dichotomy}} \label{S dichot}
	Proposition \ref{alt UGS} reveals very strong information about the density in the space around a point when there is a UGS. Armed with this knowledge, we are almost ready to prove Theorem \ref{dichotomy}. First we recall and prove two lemmas, the first of which is a technical result from \cite{Peterson2014}.
		\begin{lemma}[\cite{Peterson2014} Lemma 4.1]
		If $\b\Omega$ has a UGS, then there are constants $C_1, C_2>0$, depending only on $\Delta P$ and $\delta_0$, such that
		\begin{itemize}
			\item[(a)] ${\displaystyle
				\inf\limits_{z \in \mathbb{C}} \sup\limits_{\hat{z} \in B(z, \delta)} \sup\limits_{0 < \hat{\delta} \leq \delta} (\hat{\delta}+\hat{\delta}^2)^{-1} \int_{B(\hat{z}, \hat{\delta})} \Delta P(w)dm(w) \geq C_1, \text{ for all }\delta \geq \delta_0
			}$;
			\item[(b)] ${\displaystyle
				\sup\limits_{z \in \mathbb{C}} \ \sup\limits_{\delta>0}\  (\delta+\delta^2)^{-1} \int_{B(z,\delta)} \Delta P(w)dm(w) \leq C_2
			}$.
		\end{itemize} \label{lemma 4.1.}
	\end{lemma}
	\begin{remark} \label{lemma constants}
		Note that increasing $\delta_0$ can only possibly increase $C_1$ and will not affect the constant $C_2$.
	\end{remark}
	We also need a short geometric lemma.
	\begin{lemma}
		Let $0 < a \leq b$. Then within any disc of radius $b$ in $\mathbb{C}$, one can pack at least $\frac{b^2}{16a^2}$ disjoint discs of radius $a$. \label{ball packing}
	\end{lemma}
	\begin{proof}
		Without loss of generality, assume the disc of radius $b$ is centered at the origin. Since $B(0,a) \subset B(0,b)$, we can always pack at least one disc of radius $a$ inside of $B(0,b)$. If $2a > \sqrt{2}b$ then we have at least 1 disc of radius $a$ inside of $B(0,b)$, and
		\begin{equation*}
			1 > \frac{\sqrt{2}b}{2a} > \frac{b^2}{2a^2} > \frac{b^2}{16a^2}.
		\end{equation*}
		Note now that for all $x \geq 1$ we have $x = \lfloor x \rfloor + \alpha$ for some $\alpha \in [0,1)$ so that
		\begin{equation*}
			\lfloor x^2 \rfloor = \lfloor (\lfloor x \rfloor + \alpha)^2 \rfloor < \lfloor (\lfloor x \rfloor + \lfloor x \rfloor)^2 \rfloor = \lfloor 4\lfloor x \rfloor^2 \rfloor = 4 \lfloor x \rfloor^2.
		\end{equation*}
		Assume then that $2a \leq \sqrt{2}b$. The disc $B(0,b)$ contains a square of side length $\left\lfloor \frac{\sqrt{2}b}{2a} \right\rfloor 2a \leq \sqrt{2}b$ centered at the origin. This square contains exactly $\left\lfloor \frac{\sqrt{2}b}{2a} \right\rfloor^2$ disjoint squares of side length $2a$, each of which contains a disc of radius $a$. So we again see that $B(0,b)$ contains at least
		\begin{equation*}
			\left\lfloor \frac{\sqrt{2}b}{2a} \right\rfloor^2 > \frac{1}{4} \left\lfloor \frac{b^2}{2a^2} \right\rfloor \geq \frac{b^2}{16a^2}
		\end{equation*}
		discs of radius $a$.
	\end{proof}
	We are now ready to prove Theorem \ref{dichotomy}.
	\begin{proof}[Proof of Theorem \ref{dichotomy}]
		Proposition \ref{alt UGS} shows that there is some constant $c>0$ such that
		\begin{equation*}
			\sup_{\hat{z} \in B(z,\delta)} \sup_{0<\hat{\delta}\leq\delta} \frac{\delta}{\hat{\delta}} \int_{B(\hat{z},\hat{\delta})} \Delta P(w)dm(w) \geq cf(\delta)
		\end{equation*}
		for all $z \in \mathbb{C}$ and $\delta \geq \delta_0$. So for all $z \in \mathbb{C}$ and $\delta \geq \delta_0$, there exists $\hat{z} \in B(z,\delta)$ and $0<\hat{\delta}\leq\delta$ such that
		\begin{equation*}
			\frac{1}{\hat{\delta}} \int_{B(\hat{z},\hat{\delta})} \Delta P(w)dm(w) \geq \frac{1}{2}c\frac{f(\delta)}{\delta}.
		\end{equation*}
		Now suppose $f(\delta)=\delta$ is not a UGS for $\b\Omega$. That is, $\limsup\limits_{\delta \to +\infty} \frac{f(\delta)}{\delta} = +\infty$. Then, taking $C_2>0$ as in Lemma \ref{lemma 4.1.}, we can choose $\delta_1 > \max(1,\delta_0)$ such that $\frac{f(\delta_1)}{\delta_1} > \frac{4C_2}{c}$. Choose $\hat{\delta}$ associated to $\delta = \delta_1$ as above. If $\hat{\delta} \leq 1$, then by Lemma \ref{lemma 4.1.} we have
		\begin{equation*}
			2C_2 < \frac{c}{2} \frac{f(\delta_1)}{\delta_1} \leq \frac{1}{\hat{\delta}} \int_{B(\hat{z},\hat{\delta})} \Delta P(w)dm(w) \leq \frac{2}{\hat{\delta}+\hat{\delta}^2} \int_{B(\hat{z},\hat{\delta})} \Delta P(w)dm(w) \leq 2C_2,
		\end{equation*}
		which is impossible. Therefore for all $z \in \mathbb{C}$, there exists $\hat{z} \in B(z,\delta_1)$ and $1 \leq \hat{\delta} \leq \delta_1$ such that
		\begin{equation*}
			\int_{B(\hat{z},\hat{\delta})} \Delta P(w)dm(w) \geq \frac{c}{2}\frac{f(\delta_1)}{\delta_1} \hat{\delta} \geq 2C_2 > 0.
		\end{equation*}
		It follows that for all $z \in \mathbb{C}$
		\begin{equation*}
			\int_{B(z,2\delta_1)} \Delta P(w)dm(w) \geq \int_{B(\hat{z},\hat{\delta})} \Delta P(w)dm(w) \geq 2C_2.
		\end{equation*}
		By Lemma \ref{ball packing}, for all $\delta \geq \delta_1$, we can pack $N > \frac{\delta^2}{16\delta_1^2}$ disjoint discs of radius $2\delta_1$ within a disc of radius $2\delta$. So for all $z \in \mathbb{C}$
		\begin{equation*}
			\int_{B(z,2\delta)} \Delta P(w)dm(w) \geq N \int_{B(z,2\delta_1)} \Delta P(w)dm(w) > \frac{\delta^2}{16\delta_1^2} \cdot 2C_2 \approx (2\delta)^2.
		\end{equation*}
		Then for all $\delta \geq 2\delta_1$ and some $z_1 \in \b B(z,\delta)$
		\begin{equation*}
			f(\delta) \approx f(2\pi\delta) \approx \Lambda(z_1,2\pi\delta) \geq \int_{B(z,\delta)} \Delta P(w)dm(w) \gtrsim \delta^2.
		\end{equation*}
		But Theorem \ref{thm 1.2.} implies $f(\delta) \lesssim \delta^2$ for all $\delta \geq 2\delta_1 \geq \delta_0$. Therefore setting $\delta^*=2\delta_1$ we see that if $f(\delta)=\delta$ is not a UGS for $\b\Omega$, then $(\delta^2, \delta^*)$ is a UGS for $\b\Omega$.
	\end{proof}
	So a UGS must grow in a linear or quadratic fashion. Linear growth means that for any point, the stockyards which pick up the most mass enclose a dense, nearby disc as many times as possible. Quadratic growth means a stockyard which picks up the most mass does so by taking a pen consisting of one large disc, collecting as much area as possible.
	
\section{Identifying uniform global structures} \label{S identify}
	So far, almost all of the results of this paper have taken as hypothesis that $\b\Omega$ has a UGS and considered what that means for the global structure $\Lambda$. To look at an arbitrary model domain and determine if there is a UGS is a much more difficult question. But with Theorem \ref{dichotomy}, we see that we only need to provide conditions to identify uniform global structures where either $f(\delta)=\delta$ or $f(\delta)=\delta^2$. The following two theorems provide necessary and sufficient conditions for each case.
	\begin{theorem}
		$(\delta,\delta_0)$ is a UGS for $\b\Omega$ if and only if
		\begin{itemize}
			\item[(a)] $\displaystyle{\int_{B(z,\delta)} \Delta P(w)dm(w) \lesssim \delta}$ for all $z \in \mathbb{C}$ and $\delta > 0$, and 
			\item[(b)] There exist constants $\delta^*>M>0$ such that $\displaystyle{\inf_{z \in \mathbb{C}} \sup_{\hat{z} \in B(z,\delta^*)} \sup_{0<\hat{\delta}\leq M} \frac{1}{\hat{\delta}} \int_{B(\hat{z},\hat{\delta})} \Delta P(w)dm(w) \gtrsim 1}$.
		\end{itemize} \label{linear}
	\end{theorem}
	\begin{proof}
		Suppose $(\delta,\delta_0)$ is a UGS for $\b\Omega$. For any $z \in \mathbb{C}$, fix some point $z_1$ with $|z_1-z|=\delta$. If $2\pi\delta \geq \delta_0$ then
			\[ \int_{B(z,\delta)} \Delta P(w)dm(w) \leq \Lambda(z_1,2\pi\delta) \approx 2\pi\delta \approx \delta. \]
		If $0 < 2\pi\delta < \delta_0$, then taking a stockyard consisting of $\left\lfloor \frac{\delta_0}{2\pi\delta} \right\rfloor$ copies of $B(z,\delta)$ gives
		\begin{equation*}
			\frac{\delta_0}{4\pi\delta} \int_{B(z,\delta)} \Delta P(w)dm(w) \leq \left\lfloor \frac{\delta_0}{2\pi\delta} \right\rfloor \int_{B(z,\delta)} \Delta P(w)dm(w) \leq \Lambda(z_1,\delta_0) \approx 1.
		\end{equation*}
		Therefore (a) holds.\smallskip
		
		Also, for any fixed $\delta^* \geq \delta_0 > 0$, Lemma \ref{lemma 4.1.} gives some constant $C_1>0$ such that
		\begin{align*}
			\inf_{z \in \mathbb{C}} \sup_{\hat{z} \in B(z,\delta^*)} \sup_{0<\hat{\delta}\leq \delta_0}  \frac{1}{\hat{\delta}} \int_{B(\hat{z},\hat{\delta})} & \Delta P(w)dm(w) \\
			& \geq \inf_{z \in \mathbb{C}} \sup_{\hat{z} \in B(z,\delta_0)} \sup_{0<\hat{\delta}\leq \delta_0} \frac{1}{\hat{\delta}} \int_{B(\hat{z},\hat{\delta})} \Delta P(w)dm(w)\\
			& \geq \inf_{z \in \mathbb{C}} \sup_{\hat{z} \in B(z,\delta_0)} \sup_{0<\hat{\delta}\leq \delta_0} \frac{1}{\hat{\delta}+\hat{\delta}^2} \int_{B(\hat{z},\hat{\delta})} \Delta P(w)dm(w) \geq C_1.
		\end{align*}
		Therefore (b) holds (with $M=\delta_0$).\smallskip
		
		Now we suppose (a) and (b) hold. For any $\delta>0$ and $z \in \mathbb{C}$, let $R=(R_1,\dotsc,R_N)$ be an arbitrary $(z,\delta)$-stockyard. For each $i=1,\dotsc,N$, fix some point $z_i \in R_i$. Then recalling Remark \ref{Remark penball}, (a) gives
		\begin{equation*}
			\sum_{R_i \in R} \int_{R_i} \Delta P(w)dm(w) \leq \sum_{R_i \in R} \int_{B(z_i,\L(\b R_i))} \Delta P(w)dm(w) \lesssim \sum_{R_i \in R} \L(\b R_i) \leq \delta.
		\end{equation*}
		Therefore $\Lambda(z,\delta) \lesssim \delta$ uniformly for $z \in \mathbb{C}$ and $\delta > 0$.\smallskip
		
		For any $z \in \mathbb{C}$, fix a $\hat{z} \in B(z,\delta^*)$ and $0 < \hat{\delta} \leq M$ such that $\frac{1}{\hat{\delta}} \int_{B(\hat{z},\hat{\delta})} \Delta P(w)dm(w) \gtrsim 1$ as given by (b). Then for all $\delta \geq 2\pi M \geq 2\pi\hat{\delta}$, there is a $(z,\pi\delta^*+\delta)$-stockyard $R$ which consists of one circular pen touching $z$ and some point on $\b B(\hat{z},\hat{\delta})$ and $\left\lfloor \frac{\delta}{2\pi\hat{\delta}} \right\rfloor$ copies of $B(\hat{z},\hat{\delta})$. Then
		\begin{align*}
			\Lambda(z,\pi\delta^*+\delta) \geq \sum_{R_i \in R} \int_{R_i} \Delta P(w)dm(w) & \geq \left\lfloor \frac{\delta}{2\pi\hat{\delta}} \right\rfloor \int_{B(\hat{z},\hat{\delta})} \Delta P(w)dm(w)\\
			& \geq \frac{\delta}{4\pi\hat{\delta}} \int_{B(\hat{z},\hat{\delta})} \Delta P(w)dm(w)\\
			&\gtrsim \delta = 2\pi M \frac{\delta}{2\pi M} \geq \frac{2\pi M}{2\pi M + \pi\delta^*} (\pi\delta^*+\delta),
		\end{align*}
		where here we have used the fact that if $c \geq 0$ and $a \geq b > 0$, then $\frac{a}{b} \geq \frac{a+c}{b+c}$. Therefore $\Lambda(z,\delta) \approx \delta$ for all $\delta \geq \delta_0$ with $\delta_0 = \pi\delta^*+2\pi M$.
	\end{proof}
	\begin{theorem}
		$(\delta^2,\delta_0)$ is a UGS for $\b\Omega$ if and only if there exists $\delta^*>0$ such that, uniformly for $z \in \mathbb{C}$,  
		\begin{itemize}
			\item[(a)] $\displaystyle{\int_{B(z,\delta)} \Delta P(w)dm(w) \lesssim \delta}$ when $\delta \leq \delta^*$, and
			\item[(b)] $\displaystyle{\int_{B(z,\delta)} \Delta P(w)dm(w) \approx \delta^2}$ when $\delta \geq \delta^*$.
		\end{itemize} \label{quadratic}
	\end{theorem}
	\begin{proof}
		Suppose $(\delta^2,\delta_0)$ is a UGS for $\b\Omega$. Then for any $z \in \mathbb{C}$ and some point $z_1$ with $|z_1 - z|=\delta$ we have
		\begin{equation*}
			\int_{B(z,\delta)} \Delta P(w)dm(w) \leq \Lambda(z_1, 2\pi\delta) \approx (2\pi\delta)^2 \approx \delta^2
		\end{equation*}
		for all $\delta \geq \delta_0$.\smallskip
		
		Proposition \ref{alt UGS} shows that there is some constant $c>0$ such that
		\begin{equation*}
			\sup_{\hat{z} \in B(z, \delta)} \sup_{0<\hat{\delta}\leq\delta} \frac{\delta}{\hat{\delta}} \int_{B(\hat{z},\hat{\delta})} \Delta P(w)dm(w) \geq c\delta^2.
		\end{equation*}
		for all $z \in \mathbb{C}$ and $\delta \geq \delta_0$. So for all $z \in \mathbb{C}$ and $\delta \geq \delta_0$, there exists $\hat{z} \in B(z,\delta)$ and $0 < \hat{\delta} \leq \delta$ such that
		\begin{equation*}
			\frac{1}{\hat{\delta}} \int_{B(\hat{z},\hat{\delta})} \Delta P(w)dm(w) \geq \frac{1}{2}c\delta.
		\end{equation*}
		Taking $C_2>0$ as in Lemma \ref{lemma 4.1.}, choose some $\delta_1 > \max(1,\delta_0, \frac{4C_2}{c})$. Choose $\hat{\delta}$ associated to $\delta = \delta_1$ as above. If $\hat{\delta} \leq 1$, then by Lemma \ref{lemma 4.1.} we have
		\begin{equation*}
			2C_2 < \frac{c}{2}\delta_1 \leq \frac{1}{\hat{\delta}} \int_{B(\hat{z},\hat{\delta})} \Delta P(w)dm(w) \leq \frac{2}{\hat{\delta}+\hat{\delta}^2} \int_{B(\hat{z},\hat{\delta})} \Delta P(w)dm(w) \leq 2C_2,
		\end{equation*}
		which is impossible. Therefore for all $z \in \mathbb{C}$, there exists $\hat{z} \in B(z,\delta_1)$ and $1 \leq \hat{\delta} \leq \delta_1$ such that
		\begin{equation*}
			\int_{B(\hat{z},\hat{\delta})} \Delta P(w)dm(w) \geq \frac{c}{2}\delta_1\hat{\delta} \geq 2C_2 > 0.
		\end{equation*}
		It follows that for all $z \in \mathbb{C}$
		\begin{equation*}
			\int_{B(z,2\delta_1)} \Delta P(w)dm(w) \geq \int_{B(\hat{z},\hat{\delta})} \Delta P(w)dm(w) \geq 2C_2.
		\end{equation*}
		By Lemma \ref{ball packing}, for all $\delta \geq \delta_1$, we can pack $N > \frac{\delta^2}{16\delta_1^2}$ disjoint discs of radius $2\delta_1$ within a disc of radius $2\delta$. So for all $z \in \mathbb{C}$
		\begin{equation*}
			\int_{B(z,2\delta)} \Delta P(w)dm(w) \geq N \int_{B(z,2\delta_1)} \Delta P(w)dm(w) > \frac{\delta^2}{16\delta_1^2} \cdot 2C_2 \approx (2\delta)^2.
		\end{equation*}
		Therefore
		\begin{equation*}
			\int_{B(z,\delta)} \Delta P(w)dm(w) \approx \delta^2
		\end{equation*} for all $\delta \geq 2\delta_1 > \delta_0$. Setting $\delta^*=2\delta_1$ we see (b) holds. Also, Lemma \ref{lemma 4.1.} yields
		\begin{equation*}
			\int_{B(z,\delta)} \Delta P(w)dm(w) \leq C_2(\delta+\delta^2).
		\end{equation*}
		But if $\delta \leq \delta^*$ then
		\begin{equation*}
			\int_{B(z,\delta)} \Delta P(w)dm(w) \leq C_2(\delta^*+1)\delta \approx \delta.
		\end{equation*}
		So (a) holds.\smallskip
		
		Now we suppose (a) and (b) hold so that for $\delta \leq \delta^*$ we have $\int_{B(z,\delta)} \Delta P(w)dm(w) \leq a\delta$ and for $\delta \geq \delta^*$ we have  $\int_{B(z,\delta)} \Delta P(w)dm(w) \leq b\delta^2$ for some constants $a,b>0$. For $\delta \geq 1$, let $R=(R_1,\dotsc,R_N)$ be an arbitrary $(z,\delta)$-stockyard. Without loss of generality, me way relabel the pens so that $\L(\b R_i) \leq \delta^*$ for $i=1,\dotsc,L$ and $\L(\b R_i) \geq \delta^*$ for $i=L+1,\dotsc,N$ for some integer $L\in\{0,\dotsc,N\}$. For each $i=1,\dotsc,N$ fix some $z_i \in R_i$. Recalling Remark \ref{Remark penball}, we have
		\begin{align*}
			\sum_{R_i \in R} \int_{R_i} \Delta P(w)dm(w) & = \sum_{i=1}^L \int_{R_i} \Delta P(w)dm(w) + \sum_{i=L+1}^N \int_{R_i} \Delta P(w)dm(w)\\
			& \leq \sum_{i=1}^L \int_{B(z_i,\L(\b R_i))} \Delta P(w)dm(w) + \sum_{i=L+1}^N \int_{B(z_i,\L(\b R_i))} \Delta P(w)dm(w)\\
			& \leq a \sum_{i=1}^L \L(\b R_i) + b \sum_{i=L+1}^N \L(\b R_i)^2\\
			& \leq a \sum_{i=1}^L \L(\b R_i) + b \left( \sum_{i=L+1}^N \L(\b R_i) \right)^2
		\leq a \delta + b \delta^2 \lesssim \delta^2
		\end{align*}
		So $\Lambda(z,\delta) \lesssim \delta^2$ for all $\delta \geq 1$.\smallskip
		
		Using (b), we may take a stockyard consisting of one large circular pen with radius $\delta \geq \delta^*$ and center $z_1$ satisfying $|z_1-z|=\delta$ to see that
		\begin{equation*}
			\Lambda(z,2\pi\delta) \geq \int_{B(z_1,\delta)} \Delta P(w)dm(w) \approx \delta^2 \approx (2\pi\delta)^2.
		\end{equation*}
		Therefore $\Lambda(z,\delta) \approx \delta^2$ for all $\delta \geq \delta_0$ with $\delta_0 = \max\left(1,\frac{\delta^*}{2\pi}\right)$.
	\end{proof}
	
\section{Future directions} \label{S future}
	Although the results of this paper completely describe the nature of uniform global structures for the model domains we consider, several interesting avenues for further study present themselves when we weaken our hypotheses. One such direction would be to extend the results of this paper to higher dimensions. That is, is there an appropriate notion of stockyards in higher dimensions with which to analyze the global structure on the boundary of a model domain in $\mathbb{C}^n$? It is not clear how the Green's Theorem argument used in \cite{Peterson2014} to prove Theorem \ref{thm 1.1.} would generalize or even how (if at all) the notion of stockyards should generalize to higher dimensions.\smallskip
	
	One could also relax the conditions on $P$ which determine the boundary $\b\Omega$. For example, do similar results hold assuming that $P$ is only once differentiable and that $\Delta P$ as a distribution is non-negative? One could also allow $P$ to be a more general function for which $\Omega$ is pseudoconvex, that is take $P=P(z_1,\Re(z_2))$. In such a situation the volume of CC balls with such a choice of $P$ would \textit{a priori} depend on the $\Re(z_2)$-direction. Since the methods of this paper heavily exploited the $\Re(z_2)$-translation invariance of $\Omega$, it is unclear if these methods can be easily extended to handle the more general situation.
	
\bibliographystyle{amsalpha}

\end{document}